\documentclass[12,reqno]{amsart}
\usepackage{fullpage}
\usepackage{amsfonts}
\usepackage[active]{srcltx}

\usepackage{amsmath,amssymb,amsthm,bm}
\usepackage{mathtools}
\usepackage{shuffle}
\usepackage{bm} 
\usepackage{enumitem}
\usepackage{amsrefs}
\usepackage{empheq}
\usepackage{wasysym}
\usepackage{verbatim}
\usepackage{graphicx}
\usepackage[
bookmarks=true,         
bookmarksnumbered=true, 
colorlinks=true, pdfstartview=FitV, linkcolor=blue, citecolor=blue,
urlcolor=blue]{hyperref}
\usepackage{microtype}
\theoremstyle{plain}
\newtheorem{theorem}{Theorem}[section]

\newtheorem{lemma}[theorem]{Lemma}
\newtheorem{problem}[theorem]{Problem}
\newtheorem{proposition}[theorem]{Proposition}

\theoremstyle{remark}

\theoremstyle{definition}
\newtheorem{definition}[theorem]{Definition}

\newtheorem{assumption}[theorem]{Assumption}

\newtheorem{remark}[theorem]{Remark} \newtheorem{example}[theorem]{Example}

\newcommand{\EE}{\mathbb{E}}

\newcommand{\RR}{\mathbb{R} }

\newcommand{\D}{\mathbb D}

\newcommand{\be}{\beta}

\newcommand{\cb}{\mathcal B}

\newcommand{\cf}{\mathcal F}

\newcommand{\al}{\alpha}

\newcommand{\ga}{\gamma}

\newcommand{\si}{\sigma}

\allowdisplaybreaks[4]

\let\Section=\section
\def\section{\setcounter{equation}{0}\Section}

\def\RR{\mathbb{R} }

\def\EE{\mathbb{E}}

\def\cF{{\mathcal{F} }}

\def\si{{\sigma}}

\def\Om{{\Omega}}
\def\Ga{{\Gamma}}

\def\Om{{\Omega}}

\begin{document}

\title{Mean-field backward stochastic differential equations and applications}
\author{\textsc{Nacira AGRAM$^{1,2}$ Yaozhong HU$^{3,4,5}$ Bernt
\O KSENDAL$^{4,6}$}}
\date{12 February 2019}
\maketitle

\begin{abstract}
In this paper we study the mean-field backward stochastic differential
equations (mean-field bsde) of the form
\begin{align*}
dY(t) &  =-f(t,Y(t),Z(t),K(t,\cdot),\mathbb{E}[\varphi(Y(t),Z(t),K(t,\cdot
))])dt+Z(t)dB(t)\\
&  +{\textstyle\int_{\mathbb{R}_{0}}}
K(t,\zeta)\tilde{N}(dt,d\zeta),
\end{align*}
where $B$ is a Brownian motion, $\tilde{N}$ is the compensated Poisson random measure.
Under some mild conditions, we prove the existence and uniqueness of the solution
triplet $(Y,Z,K)$. It is commonly believed that there is no comparison theorem
for general mean-field bsde. However,
we prove a comparison theorem for a subclass of
these equations.When the  mean-field bsde
is linear, we give an explicit formula for
the first component $Y(t)$ of the solution triplet.
Our results are applied  to solve  a mean-field recursive utility optimization problem in finance.
\end{abstract}
\vskip 0.4cm



\footnotetext[1]{Department of Mathematics, Linnaeus University (LNU), V\" axj\" o, Sweden.}
\footnotetext[2]{Department of Mathematics, University Med Khider, PO Box 145, Biskra 07000, Algeria.}
\footnotetext[3]{Department of Mathematical and Statistical Sciences,
University of Alberta, Edmonton, Canada, T6G 2G1.}
\footnotetext[4]{Department of Mathematics, University of Oslo, P.O. Box 1053
Blindern, N--0316 Oslo, Norway.}

\footnotetext[5]{Supported by an NSERC grant.}
\footnotetext[6]{This research was carried out with support of the Norwegian
Research Council, within the research project Challenges in Stochastic
Control, Information and Applications (STOCONINF), project number 250768/F20.
\par
Email: agramnacira@yahoo.fr; \quad yaozhong@ualberta.ca;\quad oksendal@math.uio.no.}

\section{Introduction}




Optimal control of mean-field stochastic differential equation
 has been studied by a number of researchers
lately. To make things more precise let us explain the situation on optimal
control of stochastic systems of the following type:
\[
\left\{
\begin{array}
[c]{ll}%
dX(t) & =b(t,X(t),\mathcal{L}(X(t)), u(t))dt+\sigma(t,X(t),\mathcal{L}%
(X(t)), u(t))dB(t),\\
X(0) & =x_{0}
\end{array}
\right.
\]
with the performance
\[
J(u)=\EE \left[\int_0^T f(x(t), u(t), \mathcal{L}(X(t))) dt+
h(X(T), \mathcal{L}(X(t))) \right]\,,
\]
where $B$ is the standard $d$-dimensional
Brownian motion,  $\mathcal{L}(X(t))$ denotes the probability law  of the state
$X(t)$ at time $t$ and $b$, $\sigma$, $f$, $h$ are some  properly
  defined function.  We refer to Anderson and Djehiche \cite{AD}, \cite{AO1}, Lasry \&
Lions \cite{ll}, Carmona \& Delarue \cite{carmona}, \cite{carmona1} and Agram
\& \O ksendal \cite{A}, \cite{AO2}  for some discussion.
In particular, Pham \& Wei \cite{PW} have introduced a dynamic programming
approach by using a randomised stopping method.

Due to the presence of the law $\mathcal{L}(X(t))$ in the equation and in the performance functional, the process $X(t)$ is no longer Markovian and it is more effective to use the Pontryagin maximum principle to solve the above mean field stochastic control problem, which will give a mean-field bsde.

To limit ourselves, we shall only deal with the case
that the law  $\mathcal{L}(X(t))$ appears in the its simplest form
of expectation (see equation \eqref{e.1.1} below).
This simplest mean-field bsde also represents interesting models in finance, for example
models of risk measures and recursive utilities.

Let $c(t)\geq 0$ be a consumption rate process from a given cash flow and let
$g(t,Y(t),\mathbb{E}[Y(t)],c(t))$ be a given driver process, assumed to be
concave with respect to $Y(t),\mathbb{E}[Y(t)],c(t)$. Then the corresponding
recursive utility $U_{g}(c)$ of the consumption $c$ is the value $Y_{g}(0)$ at
$t=0$ of the first component $Y_{g}(t)$ of the solution $Y_{g},Z_{g},K_{g}$ of
the mean-field bsde
\begin{equation}
\begin{cases}
dY(t)   =-g(t,Y(t),\mathbb{E}[Y(t)],c(t))dt\\
 \qquad \qquad\qquad \qquad   +Z(t)dB(t)+
{\textstyle\int_{\mathbb{R}_{0}}}
K(t,\zeta)\tilde{N}(dt,d\zeta),t\in\left[  0,T\right]\,,    \\
Y(T)   =0\,.
\end{cases}  \label{e.1.0}
\end{equation}
The objective is to find  the consumption rate $\hat{c}$ which maximizes the mean-field
recursive utility $U_{g}(c)=Y_{g}(0)$.

This can be seen as a generalization to mean-field (and jumps) of the
classical recursive utility concept of Duffie and Epstein \cite{DE}. See also
Duffie and Zin \cite{DZ}, Kreps and Parteus \cite{KP}, El Karoui \textit{et al}
\cite{EPQ}, \O ksendal and Sulem \cite{OS3} and Agram and R\o se \cite{AR}.

Backward sde's (bsde's) were first introduced in their linear form by Bismut
\cite{b} in connection with a stochastic version of the Pontryagin maximum
principle. Subsequently, this theory was extended by Pardoux and Peng
\cite{PP} to the nonlinear case. The first work applying bsde to finance was
the paper by El Karoui \textit{et al} \cite{EPQ} where they studied several
applications to option pricing and recursive utilities. All the above
mentioned works are in the Brownian motion framework (continuous case).
The discontinuous case is more involved.
Tang and Li \cite{TL} proved an existence and uniqueness result in the case of
a natural filtration associated with a Brownian motion and a Poisson random
measure. Barles \textit{et al} \cite{BBP} proved a comparison theorem for such
equations and later Royer \cite{R} extended comparison theorem under weaker assumptions.
Buckdahn \textit{et al} \cite{BLP}, has studied a mean-field bsde and they obtained a
comparison theorem under some conditions.

In this paper, we shall study the following
mean field bsde:
\begin{equation}
\begin{cases}
dY(t)  =-f(t,y(t),z(t),k(t,\cdot),\mathbb{E[\varphi(}y(t),z(t),k(t,\cdot))])dt\\
\qquad\qquad\qquad \qquad  +Z(t)dB(t)+%
{\textstyle\int_{\mathbb{R}_{0}}}
K(t,\zeta)\tilde{N}(dt,d\zeta),t\in\left[  0,T\right]  ,\\
Y(T)   =\xi.\\
\end{cases}\label{e.1.1}
\end{equation}
The notation and conditions will be explained in details in Section 3.
The purpose of this  paper is the following.

\begin{itemize}
\item[(1)]  To prove new existence and uniqueness results for
the above mean-field bsde.

\item[(2)] To  give an explicit formula for the solution
when the equation is   linear.

\item[(3)] To  prove a comparison principle for a type of mean-field bsde different
from those in Buckdahn \textit{et al} \cite{BLP} and under weaker assumptions
on the driver.

\item[(4)] To  apply the obtained results to study a mean-field recursive utility optimization
problem in finance.
\end{itemize}
\section{Hida-Malliavin calculus}

In this section we give a brief summary of Malliavin calculus for processes
driven by Brownian motion and compensated Poisson random measures.

Let $(\Omega,\cF, \mathbb{P})$ be a probability space equipped with a
filtration $\{\cF_{t}\}_{0\le t\le T}$. The expectation on this probability
space is denoted by $\EE$ and the conditional expectation $\EE(\cdot|\cf_{t})$
is denoted by $\EE^{\cf_{t}}(\cdot)= \EE(\cdot|\cf_{t})$. Let $(B(t), 0\le
t\le T)$ be a Brownian motion.
Let $(N([0,t],\cb), 0\le t\le T, \cb\subseteq\RR_{0}=\RR-\{0\}\in\cb(\RR))$ be
a Poisson random measure. Denote by $\nu(\mathcal{B}) $ its associated L\'evy
measure so that $\mathbb{E}[N([0,t],\mathcal{B})]=\nu(\cb)t$.

Let $\tilde{N}(\cdot)$ denote the compensated Poisson measure of $N$ defined
by $\tilde{N}(dt,d\zeta):=N(dt,d\zeta)-\nu(d\zeta)dt$.
We assume that $\cf_{t}=\si(B(s),N([0,s],\cb)\,,0\leq s\leq t\,,\cb\in
\cb(\RR_{0})$. Any square integrable functional $F\in L^{2}(\Om,\cf,\mathbb{P}%
)$ can be written as
\begin{equation}
F=%
{\textstyle\sum_{m,n=0}^{\infty}}
I_{m,n}(f_{m,n})\,,
\end{equation}
where $f_{m,n}(s,t,\zeta)=f_{m,n}(s_{1},\cdots,s_{m};t_{1},\zeta_{1}%
,\cdots,t_{n}\,,\zeta_{n})$ is a function of $m+n$ variables which is
symmetric in the first $m$ variables $s=(s_{1},\cdots,s_{n})$ and the last
$n$-variables $(t,\zeta)=((t_{1},\zeta_{1}),\cdots,(t_{n},\zeta_{n}))$
satisfying
\begin{equation}%
{\textstyle\int_{\lbrack0,T]^{m+n}\times\RR^{n}}}
|f(s,t,\zeta)|^{2}ds_{1}\cdots ds_{m}dt_{1}\cdots dt_{n}\nu(d\zeta_{1}%
)\cdots\nu(d\zeta_{n})<\infty
\end{equation}
and
\begin{equation}
I_{m,n}(f_{m,n})=%
{\textstyle\int_{[0,T]^{m+n}\times\RR^{n}}}
f_{m,n}(s,t,\zeta)dB(s_{1})\cdots dB(s_{m})\tilde{N}(dt_{1},d\zeta_{1}%
)\cdots\tilde{N}(dt_{n},d\zeta_{n})
\end{equation}
is the mixed multiple integral. It is easy to see that
\begin{equation}
\EE(F^{2})=%
{\textstyle\sum_{m,n=1}^{\infty}}
m!n!%
{\textstyle\int_{[0,T]^{m+n}\times\RR^{n}}}
|f(s,t,\zeta)|^{2}ds_{1}\cdots ds_{m}dt_{1}\cdots dt_{n}\nu(d\zeta_{1}%
)\cdots\nu(d\zeta_{n})\,.
\end{equation}
We define the Malliavin derivative as $D=(D_{r}^{1},D_{\rho,\zeta}^{2})$
(where $D^{1}$ denotes the \textit{partial} Malliavin derivative with respect
to the Brownain motion and $D^{2}$ denotes the \textit{partial} Malliavin
derivative with respect to the compensated Poisson process) as follows

\begin{definition}
We say that $F$ is in $\D_{1,2}$ if
\begin{equation}%
{\textstyle\sum_{m,n=1}^{\infty}}
(m+n)m!n!%
{\textstyle\int_{[0,T]^{m+n}\times\RR^{n}}}
|f(s,t,\zeta)|^{2}ds_{1}\cdots ds_{m}dt_{1}\cdots dt_{n}\nu(d\zeta_{1}%
)\cdots\nu(d\zeta_{n})]<\infty\,.
\end{equation}
We define
\begin{align}
&  D_{r}^{1}I_{m,n}(f_{m,n})
 =mI_{m-1,n}(f_{m,n}(r,\cdot,\cdot,\cdot))\nonumber\\
&  =%
{\textstyle\int_{[0,T]^{m+n-1}\times\RR^{n}}}
f_{m,n}(s_{1},\cdots,s_{m-1},r;t,\zeta)dB(s_{1})\cdots dB(s_{m-1})\tilde
{N}(dt_{1},d\zeta_{1})\cdots\tilde{N}(dt_{n},d\zeta_{n})\,;
\end{align}
and
\begin{align}
&  D_{t,\zeta}^{2}I_{m,n}(f_{m,n})
=nI_{m-1,n}(f_{m,n}(\cdot,\cdot,(t,\zeta))\nonumber\\
&  =n%
{\textstyle\int_{[0,T]^{m+n-1}\times\RR^{n-1}}}
f_{m,n}(s_{1},\cdots,s_{m};t_{1},\zeta_{1},\cdots,t_{n-1},\zeta_{n-1}%
,t,\zeta)\nonumber\\
&dB(s_{1})\cdots dB(s_{m})\tilde{N}(dt_{1},d\zeta_{1}%
)\cdots\tilde{N}(dt_{n-1},d\zeta_{n-1})\,.
\end{align}

\end{definition}

When there is confusion, we shall also omit the superscript and write
$D_{r}=D_{r}^{1}$ and $D_{t,\zeta}=D_{t,\zeta}^{2}$.\newline The Malliavin
derivative $D_{t}$ was originally introduced by Malliavin \cite{M} as a
stochastic calculus of variation used to prove results about smoothness of
densities of solutions of stochastic differential equations in $\mathbb{R}%
^{n}$ driven by Brownian motion. The domain of definition of the Malliavin
derivative is a subspace $\mathbb{D}_{1,2}$ of $\mathbb{L}^{2}(P)$. We refer
to Stroock \cite{Stroock}, Nualart \cite{nualart}, Sanz-Sol\`{e} \cite{S}, Di Nunno \textit{et al} \cite{DOP}
and to Hu \cite{hubook} for information about the Malliavin derivative $D_{t}$
for Brownian motion and, more generally, L\'{e}vy processes. Subsequently, in
Aase \textit{et al} \cite{AaOPU} the Malliavin derivative was put into the
context of the white noise theory of Hida and extended to an operator defined
on the whole of $\mathbb{L}^{2}( \mathbb{P})$ and with values in the Hida space
$(\mathcal{S})^{\ast}$ of stochastic distributions (see below for details). It
is this extension, called the \emph{Hida-Malliavin} derivative, that we will
use in this paper. \newline There are several advantages with working with the
Hida-Malliavin derivative, as we explain in the following:

\begin{itemize}
\item The Hida-Malliavin derivative is defined on all of $\mathbb{L}^{2}( \mathbb{P})$,
and it is an extension of the classical Malliavin derivative, in the sense
that it coincides with the classical Malliavin derivative on the subspace
$\mathbb{D}_{1,2}.$

\item The Hida-Malliavin derivative combines well with the white noise
calculus, including the Skorohod integral and calculus with the Wick product
$\diamond$.
\end{itemize}

It was proved in Aase \textit{et al} \cite{AaOPU} that one can extend the
Malliavin derivative operator $D_{t}$ from $\mathbb{D}_{1,2}$ to all of
$\mathbb{L}^{2}(\mathcal{F}_{T}, \mathbb{P})$ in such a way that, also denoting the
extended operator by $D_{t}$, for all $F\in\mathbb{L}^{2}(\mathcal{F}_{T}, \mathbb{P})$,
we have
\begin{equation}
D_{t}F\in(\mathcal{S})^{\ast}\text{ and }(t,\omega)\mapsto\mathbb{E}%
[D_{t}F\mid\mathcal{F}_{t}]\text{ belongs to }\mathbb{L}^{2}(\lambda\times
 \mathbb{P}).\label{eq2.10a}%
\end{equation}
We give some properties of Hida-Malliavin derivatives. We refer to Di Nunno \textit{et al} \cite{DOP} and Agram and \O ksendal \cite {AO} for proofs and more details:

\label{eq2.2}
\begin{example}
\begin{enumerate}

\item[(i)] (Chain rule I) If $F \in L^2(\mathcal{F}_T, \mathbb{P})$ and $\varphi \in C^1 (\mathbb{R})$, then
\begin{equation}
D_t(\varphi(F)) = \varphi'(F) D_t F, \quad t \in [0,T].
\end {equation}

\item[(ii)]
(Chain rule II) If $G \in L^2(\mathcal{F}_T, \mathbb{P})$ and $\varphi \in C^1 (\mathbb{R})$, then
\begin{equation}
D_{t,\zeta}(\varphi(G)) = \varphi(G + D_{t,\zeta}G) - \varphi(G), \quad (t,\zeta) \in [0,T] \times \mathbb{R}_0.
\end {equation}

\item[(iii)]
Suppose that $F \in L^2(\mathcal{F}_t, \mathbb{P})$.
Then
$D_rF=D_{r,\zeta}F=0$ for all $r < t,\zeta \in \mathbb{R}_0$.

\item[(iv)]
Suppose $\varphi \in L^2(\lambda \times  \mathbb{P})$ is adapted and that $\psi \in L^2(\lambda \times \nu \times  \mathbb{P})$ is predictable, $\lambda$ being Lebesgue measure on $[0,T]$. Then
\begin{align*}
D_r(\int_0^T \varphi(t)dt)&= \int_r^T D_r \varphi(t)dt\\
D_r(\int_0^T \varphi (t)dB(t))&= \int_r^T D_r\varphi(t) dB(t) + \varphi(r)\\
D_{r,z}(\int_0^T \int_{\mathbb{R}} \psi(t,\zeta) \nu(d\zeta) dt)&= \int_r^T \int_{\mathbb{R}} D_{r,z} \psi(t,\zeta) \nu(d\zeta)dt\\
D_{r,z}(\int_0^T \int_{\mathbb{R}} \psi (t,\zeta)\tilde{N}(dt,d\zeta))&= \int_r^T \int_{\mathbb{R}}D_{r,z}\psi(t,\zeta) \tilde{N}(dt,d\zeta) + \psi(r,z).\\
\end{align*}

\item[(v)] Representation of bsde solution (Agram and \O ksendal \cite{AO},Theorem 2.7) : Suppose that $(p(t),q(t),r(t,\zeta))$ solves a bsde of the
form
\[
\left\{
\begin{array}
[c]{ll}%
dp(t) & = -g(t,p(t),q(t),r(t,\cdot))dt+q(t)dB(t) + \int_{\mathbb{R}_{0}} r(t,\zeta) \tilde{N}(dt,d\zeta) , 0\leq t\leq T,\\
p(T) & =F.
\end{array}
\right.
\]
Then%
\[
q(t)=D_{t^-}p(t)\text{ \ \ \ }(:=\underset{\epsilon\rightarrow0}{\lim
}D_{t-\epsilon}p(t))
\]%

and
\[
r(t,\zeta)= D_{t^{-},\zeta} p(t).
\]
\end{enumerate}
\end{example}

\section{Mean-field BSDE's}


\subsection{Existence and uniqueness of the solution}

We define the following spaces for the solution triplet:

\begin{itemize}
\item $S^{2}$ consists of the $\mathbb{F}$-adapted c\`adl\`ag processes
$Y:\Omega\times\lbrack0,T]\rightarrow\mathbb{R},$ equipped with the norm
\[
\parallel Y\parallel_{S^{2}}^{2}:=\mathbb{E[}\sup_{t\in\lbrack0,T]}%
|Y(t)|^{2}]<\infty.
\]

\item $L^{2}$ consists of the $\mathbb{F}$-predictable processes
$Z:\Omega\times\lbrack0,T]\rightarrow\mathbb{R},$ with%
\[
\parallel Z\parallel_{L^{2}}^{2}:=\mathbb{E[}%
{\textstyle\int_{0}^{T}}
\left\vert Z(t)\right\vert ^{2}dt]<\infty.
\]
{}

\item $L_{\nu}^{2}$ consists of Borel functions $K:%
\mathbb{R}
_{0}\rightarrow\mathbb{R},$ such that
\[
\parallel K\parallel_{L_{\nu}^{2}}^{2}:=%
{\textstyle\int_{\mathbb{R}_{0}}}
|K( \zeta)|^{2}\nu(d\zeta)<\infty.
\]

\item $H_{\nu}^{2}$ consists of $\mathbb{F}$-predictable
processes$\ K:\Omega\times\lbrack0,T]\times%
\mathbb{R}
_{0}\rightarrow\mathbb{R},$ such that for any fixed $t\in[0, T]$, $K(t,\zeta)$
is any element in $L^{2}_{\nu}$ and
\[
\parallel K\parallel_{H_{\nu}^{2}}^{2}:=\mathbb{E[}%
{\textstyle\int_{0}^{T}}
{\textstyle\int_{\mathbb{R}_{0}}}
K(t,\zeta)^{2}\nu(d\zeta)dt]<\infty.
\]

\item $L^{2}(\Omega,\mathcal{F}_{T})$ is the set of square integrable random
variables which are $\mathcal{F}_{T}$-measurable.
\end{itemize}

Let
\[
f:\Omega\times\lbrack0,T]\times\mathcal{\mathbb{R}}^{2}\times L_{\nu}%
^{2}\times\mathcal{%
\mathbb{R}
}^{d}\rightarrow\mathcal{%
\mathbb{R}
},
\]
be $\mathcal{F}_{t}$-progressively measurable. We consider the following
mean-field bsde:
\begin{equation}%
\begin{array}
[c]{c}%
dY(t)=f(t,Y(t),Z(t),K(t,\cdot),\mathbb{E[\varphi(}Y(t),Z(t),K(t,\cdot))])dt+Z(t)dB(t)\\
+%
{\textstyle\int_{\RR_{0}}}
K(t,\zeta)\tilde{N}(dt,d\zeta).
\end{array}
\label{b1add}%
\end{equation}

\begin{definition}
A process
\[
(Y,Z,K)\in S^{2}\times L^{2}\times H_{\nu}^{2}%
\]
is said to be a \emph{solution triplet} to the mean-field bsde \eqref{b1add}
with terminal condition $Y(T)=\xi$ if%
\[%
{\textstyle\int_{0}^{T}}
\left\vert f(s,Y(s),Z(s),K(s,\cdot),\mathbb{E[\varphi(}Y(s),Z(s),K(s,\cdot
))])\right\vert ds<+\infty\text{ }\mathbb{P}\text{-a.s.,}%
\]
and
\begin{equation}%
\begin{array}
[c]{c}%
Y(t)=\xi+%
{\textstyle\int_{t}^{T}}
f(s,Y(s-),Z(s),K(s,\cdot),\mathbb{E[\varphi(}Y(s-),Z(s),K(s,\cdot))])ds\\
-%
{\textstyle\int_{t}^{T}}
Z(s)dB(s)-%
{\textstyle\int_{t}^{T}}
{\textstyle\int_{\mathbb{R}_{0}}}
K(s,\zeta)\tilde{N}(ds,d\zeta),t\in\left[  0,T\right]  .
\end{array}
\label{b1}%
\end{equation}

where $\xi\in L^{2}(\Omega,\mathcal{F}_{T})$ is called the terminal condition
and $f$ is the generator.
\end{definition}


Strictly speaking it is $Y(s-)$ in the above equation but for simplicity we will drop the minus from now on.
To obtain the existence and uniqueness of a solution we make the following set
of assumptions.

\begin{assumption}\label{a.3.1}
For driver $f$ we assume
\begin{enumerate}
\item[(a)]   $f$ is square integrable with respect to $t$:
\[
\mathbb{E}[%
{\textstyle\int_{0}^{T}}
|f(t,0,0,0,0)|^{2}dt]<\infty\,.
\]
\item[(b)] There exists a constant $C  >0$, such that{\normalsize \ } for all {\normalsize $t\in\left[  0,T\right]  $ }and for all
$y_{1},y_{2},z_{1},z_{2}\in\mathcal{%
\mathbb{R}
},$ $k_{1},k_{2}\in L_{\nu}^{2}$ and $\mu_{1},\mu_{2}\in\mathcal{%
\mathbb{R}
}^d${\normalsize $,$ }
\begin{align*}
&  \left\vert f(t,y_{1},z_{1},k_{1},\mu_{1})-f(t,y_{2},z_{2},k_{2},\mu
_{2})\right\vert \\
&  \leq C^{\prime}(|y_{1}-y_{2}|+|z_{1}-z_{2}|+(%
\|k_1-k_2\|_{L^2(\nu)}+\left\vert \mu_{1}-\mu
_{2}\right\vert ),\text{ }\mathbb{P}\text{-a.s.}%
\end{align*}
\end{enumerate}
For the mean functional, we assume
\begin{enumerate}
\item[(c)] For each $t\in\lbrack0,T],$ the (vector valued) function $\mathbb{\varphi}%
:\Omega\times\lbrack0,T]\times\mathcal{%
\mathbb{R}
}^{2}\times L_{\nu}^{2}\rightarrow\mathcal{%
\mathbb{R}
}^d$ is assumed to be continuously differentiable with bounded partial
derivatives, such that%
\[
|\tfrac{\partial\mathbb{\varphi}}{\partial y}\mathbb{(}y,z,k)|+|\tfrac
{\partial\mathbb{\varphi}}{\partial z}\mathbb{(}y,z,k)|+||\nabla
_{k}\mathbb{\varphi(}y,z,k )||_{L_{\nu}^{2}}\leq C,
\]
for a given constant $C>0$ and $\nabla_{k}\mathbb{\varphi(}y,z,k )$ is
the Fr\'echet derivative of $\mathbb{\varphi}$ with respect to $k$.
\end{enumerate}
\end{assumption}

\begin{theorem}
\label{thm_exis} Under the assumption \ref{a.3.1}, the mean-field bsde
(\ref{b1}) has a unique solution.
\end{theorem}

\begin{proof}
For $t\in\lbrack0,T]$ and all $\beta>0,$ we introduce the norm%
\[
||(Y,Z,K)||_{\mathbb{H}_{\beta}}^{2}:=\mathbb{E[}%
{\textstyle\int_{0}^{T}}
e^{\beta t}\{|Y(t)|^{2}+|Z(t)|^{2}+%
{\textstyle\int_{\mathbb{R}_{0}}}
|K(t,\zeta)|^{2}\nu(d\zeta)\}dt].
\]
The space $\mathbb{H}_{\beta}$ equipped with this norm is an Hilbert space.
Define the mapping $\Phi:\mathbb{H}_{\beta}\rightarrow\mathbb{H}_{\beta}$ by
$\Phi(y,z,k)=(Y,Z,K)$ where $(Y,Z,K)\in S^{2}\times L^{2}\times H_{\nu}^{2}$
$(\subset L^{2}\times L^{2}\times H_{\nu}^{2})$ is defined by%
\[
\left\{
\begin{array}
[c]{ll}%
dY(t) & =-f(t,y(t),z(t),k(t,\cdot),\mathbb{E[\varphi(}y(t),z(t),k(t,\cdot))])dt\\
& +Z(t)dB(t)+%
{\textstyle\int_{\mathbb{R}_{0}}}
K(t,\zeta)\tilde{N}(dt,d\zeta),t\in\left[  0,T\right]  ,\\
Y(T) & =\xi.
\end{array}
\right.
\]
To prove the theorem it suffices to prove that $\Phi$ is contraction mapping
in $\mathbb{H}_{\beta}$ under the norm $||\cdot||_{\beta}$ for sufficiently
small $\beta$. For two arbitrary triplet $(y^{1},z^{1},k^{1}),(y^{2}%
,z^{2},k^{2})$ and $(Y^{1},Z^{1},K^{1}),(Y^{2},Z^{2},K^{2})$, we denote their
difference by $\widetilde{y}=y^{1}-y^{2}$ and $\widetilde{Y}=Y^{1}-Y^{2}$ and
similarly for $z,k,Z$ and $K$. Applying It\^o's formula to $e^{\beta
t}|\widetilde{Y}(t)|^{2}$%

\begin{align*}
& \mathbb{E[}%
{\textstyle\int_{0}^{T}}
e^{\beta t}\{\beta|\widetilde{Y}(t)|^{2}+|\widetilde{Z}(t)|^{2}+%
{\textstyle\int_{\mathbb{R}_{0}}}
|\widetilde{K}(t,\zeta)|^{2}\nu(d\zeta)\}dt]\\
&  =2\mathbb{E}[%
{\textstyle\int_{t}^{T}}
e^{\beta t}\widetilde{Y}(t)\{f(t,y^{1}(t),z^{1}(t),k^{1}(t,\cdot),\mathbb{E[\varphi
(}y^{1}(t),z^{1}(t),k^{1}(t,\cdot)])\\
&  -f(t,y^{2}(t),z^{2}(t),k^{2}(t,\cdot),\mathbb{E[\varphi(}y^{2}(t),z^{2}%
(t),k^{2}(t,\cdot))])\}dt]\;.
\end{align*}
{By the Lipschitz property of the map $f$, the mean value theorem, standard
majorization and by choosing }$\beta=1+12\overline{C}^{2}$ $(\overline{C}$
depends only on $C$ and $C^{\prime}$){, it follows that}%
\begin{align*}
&  \mathbb{E[}%
{\textstyle\int_{0}^{T}}
e^{\beta t}\{|\widetilde{Y}(t)|^{2}+|\widetilde{Z}(t)|^{2}+%
{\textstyle\int_{\mathbb{R}_{0}}}
|\widetilde{K}(t,\zeta)|^{2}\nu(d\zeta)\}dt]\\
&  \leq\tfrac{1}{2}\mathbb{E}[%
{\textstyle\int_{0}^{T}}
e^{\beta t}\{|\widetilde{y}(t)|^{2}+|\widetilde{z}(t)|^{2}+%
{\textstyle\int_{\mathbb{R}_{0}}}
|\widetilde{k}(t,\zeta)|^{2}\nu(d\zeta)\}dt],
\end{align*}
Consequently, we get
\[
||(\widetilde{Y},\widetilde{Z},\widetilde{K})||_{\beta}^{2}\leq\tfrac{1}%
{2}||(\widetilde{y},\widetilde{z},\widetilde{k})||_{\beta}^{2}\;,
\]
and $\Phi$ is then a contraction mapping. The theorem can now deduced by
standard theorem.
\end{proof}

\begin{remark}
In the above theorem if we take $d=3$, $\varphi_{i}(x_{1},x_{2},x_{3})=x_{i}$
for $i=1,2,3$, we see that the following mean-field bsde has a unique
solution
\[
\left\{
\begin{array}
[c]{ll}%
dY(t) & =-f(t,Y(t),Z(t),K(t,\cdot),\mathbb{E[}Y(t)],\mathbb{E[}Z(t)],\mathbb{E[}%
K(t,\cdot)])dt\\
& \text{ \ \ \ \ \ \ \ \ }+Z(t)dB(t)+%
{\textstyle\int_{\mathbb{R}_{0}}}
K(t,\zeta)\tilde{N}(dt,d\zeta),t\in\left[  0,T\right]  ,\\
Y(T) & =\xi,
\end{array}
\right.
\]
where $f:\Omega\times\lbrack0,T]\times\mathcal{%
\mathbb{R}
}^{2}\times L_{\nu}^{2}\times\mathcal{%
\mathbb{R}
}^{3}\rightarrow\mathcal{%
\mathbb{R}
}$ satisfies the Assumption \ref{a.3.1}.

\end{remark}

\subsection{Linear mean-field bsde}

In this section, we shall find the closed formula corresponding to the linear
mean-field bsde of the form%
\begin{equation}
\left\{
\begin{array}
[c]{ll}%
dY(t) & =-[\alpha_{1}(t)Y(t)+\beta_{1}(t)Z(t)+%
{\textstyle\int_{\mathbb{R}_{0}}}
\eta_{1}(t,\zeta)K(t,\zeta)\nu(d\zeta)+\alpha_{2}(t)\mathbb{E[}Y(t)]\\
& +\beta_{2}(t)\mathbb{E[}Z(t)]+%
{\textstyle\int_{\mathbb{R}_{0}}}
\eta_{2}(t,\zeta)\mathbb{E[}K(t,\zeta)]\nu(d\zeta)+\gamma(t)]dt\\
& +Z(t)dB(t)+%
{\textstyle\int_{\mathbb{R}_{0}}}
K(t,\zeta)\tilde{N}(ds,d\zeta),t\in\left[  0,T\right]  ,\\
Y(T) & =\xi,
\end{array}
\right.  \label{lbsde}%
\end{equation}
where the coefficients $\alpha_{1}(t),\alpha_{2}(t),\beta_{1}(t),\beta
_{2}(t),\eta_{1}(t,\cdot),\eta_{2}(t,\cdot)$ are given deterministic
functions; $\gamma(t)$ is a given $\mathbb{F}$-adapted process and $\xi\in
L^{2}\left(  \Omega,\mathcal{F}_{T}\right)  $ is a given $\mathcal{F}_{T}$
measurable random variable. Applying a result from \O ksendal and Sulem
\cite{OS3} or Quenez and Sulem \cite{qs}), the above linear mean-field bsde
$(\ref{lbsde})$ can be written as follows.
\begin{equation}%
\begin{array}
[c]{c}%
Y(t)=\mathbb{E}^{\mathcal{F}_{t}}[\xi\Gamma(t,T)+%
{\textstyle\int_{t}^{T}}
\Gamma(t,s)\{\alpha_{2}(s)\mathbb{E[}Y(s)]+\beta_{2}(s)\mathbb{E[}Z(s)]\\
+%
{\textstyle\int_{\mathbb{R}_{0}}}
\eta_{2}(t,\zeta)\mathbb{E[}K(t,\zeta)]\nu(d\zeta)+\gamma(s)\}ds],\quad
t\in\left[  0,T\right]  \,,
\end{array}
\label{cfbsde}%
\end{equation}
where $\Gamma(t,s)$ is the solution of the following linear sde
\begin{equation}
\left\{
\begin{array}
[c]{ll}%
d\Gamma(t,s) & =\Gamma(t,s^{-})[\alpha_{1}(t)dt+\beta_{1}(t)dB(t)+%
{\textstyle\int_{\mathbb{R}_{0}}}
\eta_{1}(t,\zeta)K(t,\zeta)\widetilde{N}(dt,d\zeta)],\quad s\in\left[
t,T\right]  ,\\
\Gamma(t,t) & =1\,.
\end{array}
\right.  \label{gam}%
\end{equation}
Since we are in one dimension, Equation $(\ref{gam})$ can be solved explicitly
and the solution is given by
\begin{equation}%
\begin{array}
[c]{c}%
\Gamma(t,s)=\exp\{%
{\textstyle\int_{t}^{s}}
\beta_{1}(r)dB(r)+%
{\textstyle\int_{t}^{s}}
(\alpha_{1}(r)-\tfrac{1}{2}(\beta_{1}(r))^{2})dr\\
\text{ \ \ \ \ \ \ \ \ \ \ \ \ \ \ \ \ \ }+%
{\textstyle\int_{t}^{s}}
{\textstyle\int_{\mathbb{R}_{0}}}
(\ln(1+\eta_{1}(r,\zeta))-\eta_{1}(r,\zeta))\nu(d\zeta)dr\\
\text{ \ \ \ \ }+%
{\textstyle\int_{t}^{s}}
{\textstyle\int_{\mathbb{R}_{0}}}
(\ln(1+\eta_{1}(r,\zeta))\widetilde{N}(dr,d\zeta)\}.
\end{array}
\label{gexp}%
\end{equation}
Notice that
\begin{equation}
\EE\Ga(t,s)=\exp\{%
{\textstyle\int_{t}^{s}}
\al_{1}(r)dr\}\,.
\end{equation}
To solve \eqref{cfbsde} we take the expectation on both sides of
$(\ref{cfbsde})$. Denoting $\overline{Y}(t):=\mathbb{E[}Y(t)],$ $\overline
{Z}(t):=\mathbb{E[}Z(t)]$, and $\overline{K}(t,\zeta):=\mathbb{E[}K(t,\zeta
)]$, we obtain%
\begin{equation}%
\begin{array}
[c]{c}%
\overline{Y}(t)=\mathbb{E}[\xi\Gamma(t,T)+%
{\textstyle\int_{t}^{T}}
\Gamma(t,s)\{\alpha_{2}(s)\overline{Y}(s)+\beta_{2}(s)\overline{Z}(s)\\
+%
{\textstyle\int_{\mathbb{R}_{0}}}
\eta_{2}(t,\zeta)\overline{K}(t,\zeta)\nu(d\zeta)+\gamma(s)\}ds],t\in\left[
0,T\right]  .
\end{array}
\label{e.meany}%
\end{equation}
To find equations for $\overline{Z}(t)$ and $\overline{K}(t,\zeta)$ we write
the original equation \eqref{lbsde} as a forward one:
\[%
\begin{array}
[c]{c}%
Y(t)=Y(0)+%
{\textstyle\int_{0}^{t}}
[\alpha_{1}(s)Y(s)+\alpha_{2}(s)\overline{Y}(s)+\beta_{1}(s)Z(s)+\beta
_{2}(s)\overline{Z}(s)\\
+%
{\textstyle\int_{\mathbb{R}_{0}}}
(\eta_{1}(t,\zeta)K(t,\zeta)+\eta_{2}(t,\zeta)\overline{K}(t,\zeta))\nu
(d\zeta)+\gamma(s)]ds\\
+%
{\textstyle\int_{0}^{t}}
Z(s)dB(s)+%
{\textstyle\int_{0}^{t}}
{\textstyle\int_{\mathbb{R}_{0}}}
K(s,\zeta)\tilde{N}(ds,d\zeta),\quad t\in\left[  0,T\right]  ,
\end{array}
\]
for some deterministic initial value $Y(0)$. Then using the properties stated in Example 2.2, we compute the Hida-Malliavin derivative of $Y(t)$ for all $r<t$ as follows:
\begin{align*}
D_{r}Y(t) &  =%
{\textstyle\int_{r}^{t}}
D_{r}[\alpha_{1}(s)Y(s)+\alpha_{2}(s)\overline{Y}(s)+\beta_{1}(s)Z(s)+\beta
_{2}(s)\overline{Z}(s)\\
&  +%
{\textstyle\int_{\mathbb{R}_{0}}}
(\eta_{1}(s,\zeta)K(s,\zeta)+\eta_{2}(s,\zeta)\overline{K}(s,\zeta))\nu
(d\zeta)+\gamma(s)]ds\\
& +
{\textstyle\int_{r}^{t}}
D_r  Z(s)dB(s)+Z(r).
\end{align*}
Letting $r\rightarrow t-$, we get that $Z(t)=D_{t}Y(t).$ Thus, to find $Z(t)$
we only need to compute $D_{t}Y(t)$. We shall use the expression
\eqref{cfbsde} for $Y(t)$ and the identity
\[
D_{t}\EE^{\cf_{t}}[F]=\EE^{\cf_{t}}[D_{t}F]\,.
\]
We also notice that $D_{t}\Gamma(t,T)=\Gamma(t,T)\beta_{1}(t)$. Then
\begin{align*}
Z(t) &  =\mathbb{E}^{\mathcal{F}_{t}}[D_{t}\xi\Gamma(t,T)+\xi\Gamma
(t,T)\beta_{1}(t)+%
{\textstyle\int_{t}^{T}}
\Gamma(t,s)\beta_{1}(t)\{\alpha^{2}(s)\overline{Y}(s)\\
&  \qquad+\beta_{2}(s)\overline{Z}(s)+%
{\textstyle\int_{\mathbb{R}_{0}}}
\eta_{2}(s,\zeta)\overline{K}(s,\zeta)\nu(d\zeta)+\gamma(s)\}ds]\,.
\end{align*}
Taking the expectation, we have
\begin{align}
\overline{Z}(t) &  =\mathbb{E}[D_{t}\xi\Gamma(t,T)+\beta_{1}(t)\mathbb{E(}%
\xi\Gamma(t,T))+%
{\textstyle\int_{t}^{T}}
\mathbb{E(}\Gamma(t,s))\beta_{1}(t)\{\alpha_{2}(s)\overline{Y}(s)\nonumber\\
&  \qquad+\beta_{2}(s)\overline{Z}(s)+%
{\textstyle\int_{\mathbb{R}_{0}}}
\eta_{2}(s,\zeta)\overline{K}(s,\zeta)\nu(d\zeta)+\gamma
(s)\}ds].\label{e.meanz}%
\end{align}
Similarly, we have $K(t,\zeta)=D_{t,\zeta}Y(t)$ which yields
\begin{align*}
K(t,\zeta) &  =\mathbb{E}^{\mathcal{F}_{t}}[D_{t,\zeta}\xi\Gamma
(t,T)+\xi\Gamma(t,T)\eta_{1}(t,\zeta)+%
{\textstyle\int_{t}^{T}}
\Gamma(t,s)\eta_{1}(t,\zeta)\{\alpha_{2}(s)\overline{Y}(s)\\
&  \qquad+\beta_{2}(s)\overline{Z}(s)+%
{\textstyle\int_{\mathbb{R}_{0}}}
\eta_{2}(s,\zeta)\overline{K}(s,\zeta)\nu(d\zeta)+\gamma(s)\}ds]\,.
\end{align*}
Taking the expectation yields
\begin{equation}
\overline{K}(t,\zeta)=\mathbb{E}[D_{t,\zeta}\xi\Gamma(t,T)+\xi\Gamma
(t,T)\eta_{1}(t,\zeta)+%
{\textstyle\int_{t}^{T}}
\eta_{2}(s,\zeta)\overline{K}(s,\zeta)\nu(d\zeta)+\gamma
(s)\}ds]\,.\label{e.meank}%
\end{equation}
Equations \eqref{e.meany}, \eqref{e.meanz} and \eqref{e.meank} can be used to
obtain $\bar{Y},\bar{Z},\bar{K}$. In fact, we let
\[
V(t)=\left(
\begin{array}
[c]{c}%
V_{1}(t)\\
V_{2}(t)\\
V_{3}(t,\zeta)
\end{array}
\right)  =\left(
\begin{array}
[c]{c}%
\overline{Y}(t)\\
\overline{Z}(t)\\
\overline{K}(t,\zeta)
\end{array}
\right)  \in L^{2}\times L^{2}\times H_{\nu}^{2},
\]
and
\begin{align}
A(t,s,\zeta) &  =\left(  A_{ij}(t,s,\zeta)\right)  _{1\leq i,j\leq3}%
\label{A}\\
&  =\left(
\begin{matrix}
\exp\{\int_{t}^{s}\al_{1}(r)dr\}\al_{2}(s) & \exp\{\int_{t}^{s}\al_{1}%
(r)dr\}\be_{2}(s) & \exp\{\int_{t}^{s}\al_{1}(r)dr\}\eta_{2}(s,\zeta)\\
\exp\{\int_{t}^{s}\al_{1}(r)dr\}\beta_{1}(t)\al_{2}(s) & \exp\{\int_{t}%
^{s}\al_{1}(r)dr\}\beta_{1}(t)\be_{2}(s) & \exp\{\int_{t}^{s}\al_{1}%
(r)dr\}\beta_{1}(t)\eta_{2}(s,\zeta)\\
\exp\{\int_{t}^{s}\al_{1}(r)dr\}\eta_{1}(t,\zeta)\al_{2}(s) & \exp\{\int
_{t}^{s}\al_{1}(r)dr\}\eta_{1}(t,\zeta)\be_{2}(s) & \exp\{\int_{t}^{s}%
\al_{1}(r)dr\}\eta_{1}(t,\zeta)\eta_{2}(s,\zeta)\
\end{matrix}
\right)  \,.\nonumber
\end{align}
Define a mapping $A=A^{T}$ from $V=(V_{1},V_{2},V_{3})^{T}\in L^{2}\times L^{2}\times
H_{\nu}^{2}$ to itself by
\begin{equation}
(AV)_{i}(t,\zeta)=
{\textstyle\sum_{j=1}^{2}}
{\textstyle\int_{t}^{T}}
A_{ij}(t,s)V_{j}(s)ds+%
{\textstyle\int_{t}^{T}}
{\textstyle\int_{\RR_{0}}}
A_{i3}(t,s,\zeta)V_{3}(s,\zeta)\nu(d\zeta)\,ds.
\end{equation}
Then \eqref{e.meany}, \eqref{e.meanz} and \eqref{e.meank} can be
written as
\begin{equation}
V=F+AV\,,\label{e.v}%
\end{equation}
where
\begin{equation}
F(t,\zeta)=\left(
\begin{matrix}
\EE(\xi\Ga(t,T))+\int_{t}^{T}\ga(s)ds\\
\EE[D_{t}\xi\Ga(t,T)+\beta^{1}(t)\xi\Ga(t,T)]+\int_{t}^{T}\ga(s)ds\\
\EE[D_{t,\zeta}\xi\Ga(t,T)+\xi\Ga(t,T)\eta^{1}(t,\zeta)]+\int_{t}^{T}\ga(s)ds
\end{matrix}
\right)  \,.\label{F}%
\end{equation}
Note that the operator norm of $A$, $||A||$, is less than 1 if $t$ is close enough to $T$. Therefore there exists $\delta > 0$ such that $||A|| < 1$ if we restrict the operator to the interval $[T-\delta,T]$ for some $\delta >0$ small enough.
In this case the linear equation equation \eqref{e.v} can now be solved easily as follows:
\[
(I-A)V=F\,,
\]
or
\begin{equation}
V=(I-A)^{-1}F=%
{\textstyle\sum_{n=0}^{\infty}}
A^{n}F\, ; \quad t \in [T-\delta,T].
\end{equation}
Next, using  $V(T-\delta)$ as the terminal value of the corresponding BSDE in the interval $[T-2\delta,T-\delta]$ and repeating the argument above, we find that there exists a solution $V$ of the BSDE in this interval, given by the equation
\begin{equation}
V(t,\zeta) = V(T-\delta,\zeta) + A ^{T-\delta}(t,\cdot,\zeta) V(\cdot) ;\quad T-2\delta \leq t \leq T-\delta.
\end{equation}
Proceeding by induction we end up with a solution on the whole interval $[0,T]$. We summarise this as follows:


\begin{theorem}
[Closed formula] Assume that $\alpha_{1}(t),\alpha_{2}(t),\beta_{1}(t),\beta
_{2}(t),\eta_{1}(t,\cdot),\eta_{2}(t,\cdot)$ are given bounded deterministic functions
and that $\gamma(t)$ is $\mathbb{F}$-adapted and $\xi\in L^{2}\left(
\Omega,\mathcal{F}_{T}\right)  $. Then the component $Y(t)$ of the solution of the
linear mean-field bsde (\ref{lbsde}) can be written on its closed formula as
follows%
\begin{equation}
Y(t)=\mathbb{E}^{\mathcal{F}_{t}}[\xi\Gamma(t,T)+%
{\textstyle\int_{t}^{T}}
\Gamma(t,s)\{(\alpha_{2}(s),\beta_{2}(s),\eta_{2}(s,\zeta))V(s)+\gamma
(s)\}ds],t\in\left[  0,T\right]  ,\text{ }\mathbb{P}\text{-a.s.,}\label{cf_y}%
\end{equation}
where
\[%
\begin{array}
[c]{c}%
\Gamma(t,s)=\exp\{%
{\textstyle\int_{t}^{s}}
\beta_{1}(r)dB(r)+%
{\textstyle\int_{t}^{s}}
(\alpha_{1}(r)-\tfrac{1}{2}(\beta_{1}(r))^{2})dr\\
\text{ \ \ \ \ \ \ \ \ \ \ \ \ \ \ \ \ \ }+%
{\textstyle\int_{t}^{s}}
{\textstyle\int_{\mathbb{R}_{0}}}
(\ln(1+\eta_{1}(r,\zeta))-\eta_{1}(r,\zeta))\nu(d\zeta)dr\\
\text{ \ \ \ \ }+%
{\textstyle\int_{t}^{s}}
{\textstyle\int_{\mathbb{R}_{0}}}
(\ln(1+\eta_{1}(r,\zeta))\widetilde{N}(dr,d\zeta)\}.
\end{array}
\]
and, inductively,
\begin{equation}
V(t,\zeta) = V(T-k\delta,\zeta) + A ^{T-k\delta}(t,\cdot,\zeta) V(\cdot) ;\quad T-(k+1)\delta \leq t \leq T-k\delta; \quad k=0,1,2, ...
\end{equation}
Or, equivalently,
\begin{align}
V(t,\zeta)=
{\textstyle\sum_{n=0}^{\infty}}
(A ^{T-k\delta}(t,\cdot,\zeta))^{n}V(T-k\delta,\cdot) ;\quad T-(k+1)\delta \leq t \leq T-k\delta; \quad k=0,1,2, ...
\end{align}
where $A^{S}; S > 0$ is given by (\ref{A}) and $V(T,\zeta)=F$.
\end{theorem}

\section{A comparison theorem for mean-field bsde's}

In this section we are interested in a subclass of mean-field bsde. Our idea
is to use Picard iteration. So first, we shall prove a convergence result for
the Picard iteration.

\subsection{Picard iteration}

To be able to prove the comparison theorem for mean-field bsde, we consider a
mean field bsde and with driver allowed only to depend on the expectation of
$Y(t)$ and independent of the expectations of $Z(t)$ and $K(t,\zeta)$, as follows%

\begin{equation}
\left\{
\begin{array}
[c]{ll}%
dY(t) & =-g(t,Y(t),Z(t),K(t,\cdot),\mathbb{E[}Y(t)\mathbb{]})dt+Z(t)dB(t)\\
& \text{ \ \ \ \ \ \ \ \ \ \ }+%
{\textstyle\int_{\mathbb{R}_{0}}}
K(t,\zeta)\tilde{N}(dt,d\zeta),t\in\left[  0,T\right]  ,\\
Y(T) & =\xi.
\end{array}
\right.  \label{p_bsde}%
\end{equation}
We impose the following set of assumptions.

\begin{assumption}\label{a.4.1}
\begin{itemize}
\item[(i)] Here $g:\Omega\times\lbrack0,T]\times\mathcal{%
\mathbb{R}
}^{2}\times L_{\nu}^{2}\times\mathcal{%
\mathbb{R}
}\rightarrow\mathcal{%
\mathbb{R}
}$ is $\mathbb{F}$-adapted and satisfies the Lipschitz assumption in the sense
that%
\begin{align*}
|g(t,y,z,k,\overline{y})-g(t,y^{\prime},z^{\prime},k^{\prime},\overline
{y}^{\prime})|     \leq C(|y-y^{\prime}|+|z-z^{\prime}|  +
\|k-k'\|_{L^2(\nu)}+|
\overline{y}-\overline{y}^{\prime}|),
\end{align*}
for all $y,z,\overline{y},y^{\prime},z^{\prime},\overline{y}^{\prime}
\in
\mathbb{R} ,k,k^{\prime}\in L_{\nu}^{2}$.
\item[(ii)]
\[
\mathbb{E}[  \int_{0}^{T}
|g(t,0,0,0,0)|^{2}dt]<\infty\,.
\]
\item[(iii)] The terminal value $\xi\in L^{2}\left(  \Omega,\mathcal{F}_{T}\right)  $.
\end{itemize}
\end{assumption}
\vskip 0.2cm
The following result is a consequence of Theorem ${3.3}$ with ${d=1}$ and
$\varphi(x)=x:$

\begin{theorem}
Under the above Assumption 4.1, the mean-field bsde
(\ref{p_bsde}) admits a unique solution $(Y,Z,K)\in S^{2}\times L^{2}\times
H_{\nu}^{2}.$
\end{theorem}

To prove a comparison theorem, we need the following
convergence to hold:

\begin{lemma}
[Convergence]\label{conv} Let $(Y,Z,K)\in S^{2}\times L^{2}\times H_{\nu}^{2}$
satisfies the mean-field bsde
\begin{equation}%
\begin{array}
[c]{c}%
Y(t)=\xi+%
{\textstyle\int_{t}^{T}}
g(s,Y(s),Z(s),K\left(  s,\cdot\right)  ,\mathbb{E[}Y(s)\mathbb{]})ds-%
{\textstyle\int_{t}^{T}}
Z(s)dB(s)\\
-%
{\textstyle\int_{t}^{T}}
{\textstyle\int_{\mathbb{R}_{0}}}
K(s,\zeta)\tilde{N}(ds,d\zeta),t\in\left[  0,T\right]  ,
\end{array}
\label{y}%
\end{equation}
where $\xi$ and $g$ are supposed to satisfy Assumption 4.1. We assume that
for all $n\geq 1,$ the triplet $(Y^{n},Z^{n},K^{n})$ satisfies%
\begin{equation}%
\begin{array}
[c]{c}%
Y^{n}(t)=\xi+\int_{t}^{T}g(s,Y^{n}(s),Z^{n}(s),K^{n}(s,\cdot),\mathbb{E}%
[Y^{n-1}(s)])ds-\int_{t}^{T}Z^{n}(s)dB(s)\\
-%
{\textstyle\int_{t}^{T}}
{\textstyle\int_{\mathbb{R}_{0}}}
K^{n}(s,\zeta)\tilde{N}(ds,d\zeta),t\in\left[  0,T\right]  ,
\end{array}
\label{yn}%
\end{equation}
where $Y^{n-1}(s)$ is known. Thus, the following convergence holds%
\[
Y^{n}(t)\rightarrow Y(t)\text{, for each }t\in\left[  0,T\right]  \text{.}%
\]

\end{lemma}

\noindent{Proof.} \quad\ The proof relies on the classical Picard iteration
method.\newline Define $Y^{0}(t)=\mathbb{E}[\xi]$ and $Y^{n}(t)$ given by
(\ref{yn}) inductively as follows:%
\[%
\begin{array}
[c]{c}%
Y^{n+1}(t)=\xi+%
{\textstyle\int_{t}^{T}}
g(s,Y^{n+1}(s),Z^{n+1}(s),K^{n+1}(s,\cdot),\mathbb{E}[Y^{n}(s)])ds-%
{\textstyle\int_{t}^{T}}
Z^{n+1}(s)dB(s)\\
-%
{\textstyle\int_{t}^{T}}
{\textstyle\int_{\mathbb{R}_{0}}}
K^{n+1}(s,\zeta)\tilde{N}(ds,d\zeta),t\in\left[  0,T\right]  .
\end{array}
\]
We want to show that the sequence $(Y^{n}(t))_{t\geq0}$ forms a Cauchy
sequence. By It\^{o}'s formula, we have%
\[%
\begin{array}
[c]{l}%
\mathbb{E}[|Y^{n+1}(t)-Y^{n}(t)|^{2}]+\tfrac{1}{2}\mathbb{E}[%
{\textstyle\int_{t}^{T}}
|Y^{n+1}(s)-Y^{n}(s)|^{2}ds]+\tfrac{1}{2}\mathbb{E}[%
{\textstyle\int_{t}^{T}}
|Z^{n+1}(s)-Z^{n}(s)|^{2}ds]\\
\text{ \ \ \ \ \ \ \ \ \ \ \ \ \ \ \ \ \ \ \ \ }+\tfrac{1}{2}\mathbb{E}%
{\textstyle\int_{t}^{T}}
{\textstyle\int_{\mathbb{R}_{0}}}
[|K^{n+1}(s,\zeta)-K^{n}(s,\zeta)|^{2}]\nu(d\zeta)ds\\
\leq C\mathbb{E}[\int_{t}^{T}|Y^{n+1}(s)-Y^{n}(s)|^{2}ds]+\mathbb{E}[%
{\textstyle\int_{t}^{T}}
|Y^{n}(s)-Y^{n-1}(s)|^{2}ds].
\end{array}
\]
This implies%
\[
-\tfrac{d}{dt}(e^{Ct}\mathbb{E}[|Y^{n+1}(t)-Y^{n}(t)|^{2}])\leq\tfrac{1}%
{2}\mathbb{E}[%
{\textstyle\int_{t}^{T}}
|Y^{n}(s)-Y^{n-1}(s)|^{2}ds].
\]
Integrating both sides from $u$ to $T$, yields%
\begin{align*}%
{\textstyle\int_{u}^{T}}
\mathbb{E}[|Y^{n+1}(t)-Y^{n}(t)|^{2}])dt  &  \leq\tfrac{1}{2}%
{\textstyle\int_{t}^{T}}
dte^{C(t-u)}\mathbb{E}[%
{\textstyle\int_{t}^{T}}
|Y^{n}(s)-Y^{n-1}(s)|^{2}ds]\\
&  \leq e^{CT}%
{\textstyle\int_{t}^{T}}
dt\mathbb{E}[%
{\textstyle\int_{t}^{T}}
|Y^{n}(s)-Y^{n-1}(s)|^{2}ds].
\end{align*}
By induction on $n$, we get%
\[
\mathbb{E}[%
{\textstyle\int_{0}^{T}}
|Y^{n+1}(t)-Y^{n}(t)|^{2}dt])\leq\tfrac{e^{CnT}T^{n}}{n!}.
\]
We conclude that there exists a unique $\mathbb{F}$-adapted process $Y(t)$
such that $Y^{n}(t)$ converges to $Y(t)$ which satisfies equation
(\ref{y}).$\qquad\qquad\square$ \newline\newline We are now ready to state and
prove a comparison theorem for mean-field bsde .

\begin{theorem}
[Comparison Theorem]\label{comp} Let $g_{1},g_{2}:\Omega\times\lbrack
0,T]\times\mathcal{%
\mathbb{R}
}^{2}\times L_{\nu}^{2}\times\mathcal{%
\mathbb{R}
}$ and $\xi_{1},\xi_{2}\in L^{2}\left(  \Omega,\mathcal{F}_{T}\right)  $ and
let $(Y_{i},Z_{i},K_{i})_{i=1,2}$ be the solutions of the following mean-field
bsde's
\[%
\begin{array}
[c]{c}%
Y_{i}(t)=\xi_{i}+%
{\textstyle\int_{t}^{T}}
g_{i}\left(  s,Y_{i}\left(  s\right)  ,Z_{i}\left(  s\right)  ,K_{i}\left(
s,\cdot\right)  ,\mathbb{E}[Y_{i}\left(  s\right)  ]\right)  ds-%
{\textstyle\int_{t}^{T}}
Z_{i}(s)dB(s)\\
-%
{\textstyle\int_{t}^{T}}
{\textstyle\int_{\mathbb{R}_{0}}}
K_{i}(s,\zeta)\tilde{N}(ds,d\zeta),t\in\left[  0,T\right]  .
\end{array}
\]

\begin{equation}
\xi_{1}\geq\xi_{2}\text{ }\mathbb{P}\text{-a.s.} \label{xi_est}%
\end{equation}

Assume that the drivers $(g_{i})_{i=1,2}$ are given $\mathbb{F}$-predictable
processes satisfying Assumption \ref{a.4.1} and
\begin{equation}
g_{1}\left(  t,y_{1},z_{1},k_{1} ,\overline{y}_{1}\right)  \geq g_{2}\left(
t,y_{1},z_{1},k_{1} ,\overline{y}_{2}\right)  ,\forall t\text{, }\overline
{y}_{1}\geq\overline{y}_{2},\mathbb{P}\text{-a.s.,} \label{g_est}%
\end{equation}
and moreover, the following inequality holds
\begin{equation}
g_{2}(t,y,z,k_{1} ,\mathbb{E}[Y(t)])-g_{2}(t,y,z,k_{2} ,\mathbb{E}[Y(t)]) \geq%
{\textstyle\int_{\mathcal{\mathbb{R}}_{0}}}
\eta^{1}(t,\zeta)(k_{1}(\zeta)-k_{2}(\zeta))\nu(d\zeta), \label{jump_est}%
\end{equation}
$\mathbb{P}$-a.s. for all $t$.
\newline

Then $Y_{1}(t)\geq Y_{2}(t)$ $\mathbb{P}$-a.s. for each $t$.\newline
\end{theorem}

\begin{proof}
We use Picard iteration and we shall prove that $Y_{1}^{n}(t)\geq
Y_{2}^{n}(t)$ for all $n$ and $t$ by using induction on $n$. Let $Y_{1}%
^{0}(t)=\mathbb{E}[\xi_{1}]$ and $Y_{2}^{0}(t)=\mathbb{E}[\xi_{2}].$ Then
\[
Y_{1}^{0}(t)\geq Y_{2}^{0}(t)\text{, for each }t\geq0.
\]
Assume
\[
Y_{1}^{n}(t)\geq Y_{2}^{n}(t)\text{, for each }t\geq0.
\]
Define the triple $(Y_{i}^{n+1}(t),Z_{i}^{n+1}(t),K_{i}^{n+1}(t,\cdot
))_{i=1,2}$, as follows
\[%
\begin{array}
[c]{c}%
Y_{i}^{n+1}(t)=\xi+%
{\textstyle\int_{t}^{T}}
g_{i}(s,Y_{i}^{n+1}(s),Z_{i}^{n+1}(s),K_{i}^{n+1}(s,\cdot),\mathbb{E}%
[Y_{i}^{n}(s)])ds-%
{\textstyle\int_{t}^{T}}
Z_{i}^{n+1}(s)dB(s)\\
-%
{\textstyle\int_{t}^{T}}
{\textstyle\int_{\mathbb{R}_{0}}}
K^{n+1}(s,\zeta)\tilde{N}(ds,d\zeta),t\in\left[  0,T\right]  ,
\end{array}
\]
where $Y^{n}(t)$ is knowing. Define $\bar{g}_{i}(t,y,z,k):=g_{i}%
(t,y,z,k,\mathbb{E}[Y_{i}^{n}(t)])$, then
\[%
\begin{array}
[c]{c}%
Y_{i}^{n+1}(t)=\xi+%
{\textstyle\int_{t}^{T}}
\bar{g}_{i}(s,Y_{i}^{n+1}(s),Z_{i}^{n+1}(s),K_{i}^{n+1}(s,\cdot))ds-%
{\textstyle\int_{t}^{T}}
Z_{i}^{n+1}(s)dB(s)\\
-%
{\textstyle\int_{t}^{T}}
{\textstyle\int_{\mathbb{R}_{0}}}
K^{n+1}(s,\zeta)\tilde{N}(ds,d\zeta),t\in\left[  0,T\right]  .
\end{array}
\]
We have by our assumptions that
\[
\bar{g}_{1}(t,y,z,k)\geq\bar{g}_{2}(t,y,z,k),\text{ for each }t\geq0.
\]
{By} the comparison theorem for BSDE with jumps e.g. Theorem 2.3 in Royer \cite{R}, it follows that $Y_{1}^{n+1}(t)\geq
Y_{2}^{n+1}(t)$ for all $t\geq0$.\newline By our convergence result
\ref{conv}, we conclude that%
\[
Y_{1}(t)\geq Y_{2}(t),\text{ for each }t\geq0.
\]

\end{proof}

\section{Mean-field recursive utility}

We consider in this section a mean-field recursive utility process $Y(t)$, defined to be the first component of the
solution triplet $(Y,Z,K)$ of the following mean-field bsde:

\begin{equation}
\left\{
\begin{array}
[c]{ll}%
dY(t) & =-g(t,Y(t),Z(t),K(t,\cdot),\mathbb{E[}Y(t)\mathbb{]},\mathbb{E[}%
Z(t)\mathbb{]},\mathbb{E[}K(t,\cdot)\mathbb{]},\pi(t))dt\\
& \text{ \ \ \ \ \ \ \ \ \ \ }+Z(t)dB(t)+%
{\textstyle\int_{\mathbb{R}_{0}}}
K(t,\zeta)\tilde{N}(dt,d\zeta),t\in\left[  0,T\right]  ,\\
Y(T) & =\xi.
\end{array}
\right.  \label{ru_bsde}%
\end{equation}
We denote by $\mathcal{U}$, the set of all consumption processes. For each
$\pi(t)\in\mathcal{U}$, the driver $g:\Omega\times\lbrack0,T]\times\mathcal{%
\mathbb{R}
}^{2}\times L_{\nu}^{2}\times\mathcal{%
\mathbb{R}
}^{2}\times L_{\nu}^{2}\times\mathcal{U}\rightarrow\mathcal{%
\mathbb{R}
}$ and the terminal value $\xi$ satisfies assumptions (I). Suppose that
$(y,z,k,\overline{y},\overline{z},\overline{k},\pi)\mapsto g(t,y,z,k,\bar
{y},\overline{z},\overline{k},\pi)$ is concave for each $t\in\lbrack0,T]$. The
driver
\[
g(t,Y(t),Z(t),K(t,\cdot),\mathbb{E[}Y(t)\mathbb{]},\mathbb{E[}Z(t)\mathbb{]}%
,\mathbb{E[}K(t,\cdot)\mathbb{]},\pi(t))
\]
represents the instantaneous utility at time $t$ of the consumption rate
$\pi(t)\geq0$, such that
\[
\mathbb{E[}%
{\textstyle\int_{0}^{T}}
|g(t,0,0,0,0,\pi(t))|^{2}dt]<\infty\text{, for all }t\in\left[  0,T\right]  .
\]
\newline We call a process $\pi(t)$ a consumption rate process if $\pi(t)$ is
predictable and $\pi(t)\geq0$ for each $t$ $\mathbb{P}$-a.s. Then
$Y(t)=Y_{g}(0)$ is called a mean-field recursive utility process of the
consumption $\pi(\cdot)$, and the number $U(\pi)=Y_{g}(0)$ is called the total
mean-field recursive utility of $\pi(\cdot)$. This is an extension to
mean-field (and jumps) of the classical recursive utility concept of Duffie
and Epstein \cite{DE}. See also Duffie and Zin \cite{DZ}, Kreps and Parteus
\cite{KP}, El Karoui et al \cite{EPQ}, \O ksendal and Sulem \cite{OS3} and
Agram and R\o se \cite{AR} and the reference their in. Finding the consumption
rate $\hat{\pi}$ which maximizes its total mean-field recursive utility is an
interesting problem in mean-field stochastic control.\newline

\subsection{Optimization problem}

We discuss now the optimization problem related to the recursive utility. The
wealth process $X(t)=X^{\pi}(t)$ is given by the following linear
sde{\normalsize
\begin{equation}
\left\{
\begin{array}
[c]{ll}%
dX(t) & =[b_{0}(t)-\pi(t)]X(t)dt+\sigma_{0}(t)X(t)dB(t)\\
& +%
{\textstyle\int_{\mathbb{R}_{0}}}
\gamma_{0}(t,\zeta)X(t)\tilde{N}(dt,d\zeta),t\in\lbrack0,T],\\
X(0) & =x_{0},
\end{array}
\right.  \label{SDE}%
\end{equation}
}where the initial value $x_{0}>0$, and the functions $b_{0},$ $\sigma_{0}$,
$\gamma_{0}$ are assumed to be deterministic functions, $\pi$ is our relative
consumption rate at time $t$, assumed to be a c\`{a}dl\`{a}g $\mathbb{F}%
$-adapted process. We assume that $%
{\textstyle\int_{0}^{T}}
\pi(t)dt<\infty$ $\mathbb{P}$-a.s. This implies that our wealth process
$X(t)>0$ for all $t$ $\mathbb{P}$-a.s. Define the recursive utility process
$Y(t)=Y^{\pi}(t)$ by the linear mean-field bsde in the unknown triplet
$(Y,Z,K)=(Y^{\pi},Z^{\pi},K^{\pi})\in S^{2}\times L^{2}\times H_{\nu}^{2}$,
by{\normalsize
\begin{equation}
\left\{
\begin{array}
[c]{ll}%
dY(t) & =-[\alpha_{0}(t)Y(t)+\alpha_{1}(t)\mathbb{E[}Y(t)]+\beta
_{0}(t)Z(t)+\beta_{1}(t)\mathbb{E[}Z(t)]\\
& +%
{\textstyle\int_{\mathbb{R}_{0}}}
\{\eta_{0}(t,\zeta)K(t,\zeta)+\eta_{1}(t,\zeta)\mathbb{E[}K(t,\zeta
)]\}\nu(d\zeta)+\ln(\pi(t)X(t)\text{)}]dt\\
& +Z(t)dB(t)+%
{\textstyle\int_{\mathbb{R}_{0}}}
K(t,\zeta)\tilde{N}(dt,d\zeta),t\in\left[  0,T\right]  ,\\
Y(T) & =\theta X(T),
\end{array}
\right.  \label{ru_bsde}%
\end{equation}
}where $\theta=\theta(\omega)>0$ is a given bounded random variable and
$\alpha_{0},\alpha_{1},\beta_{0},\beta_{1},\eta_{0},\eta_{1}$ are given
deterministic functions with $\eta_{0}(t,\zeta),\eta_{1}(t,\zeta)\geq
-1.$\newline From the closed formula (\ref{cf_y}), the first component $Y(t)$ of the
solution triplet of the equation (\ref{ru_bsde}) can be written as
\begin{align*}
&
\begin{array}
[c]{c}%
Y(t)=\mathbb{E}^{\mathcal{F}_{t}}[\theta X(T)\Gamma(t,T)
\end{array}
\\
&
\begin{array}
[c]{c}%
+%
{\textstyle\int_{t}^{T}}
\Gamma(t,s)\{(\alpha_{1}(s),\beta_{1}(s),\eta_{1}(s,\zeta))V(s)+\ln
(\pi(s)X(s))\}ds],t\in\left[  0,T\right]
\end{array}
\end{align*}
where
\[%
\begin{array}
[c]{c}%
\Gamma(t,s)=\exp\{%
{\textstyle\int_{t}^{s}}
\beta_{0}(r)dB(r)+%
{\textstyle\int_{t}^{s}}
(\alpha_{0}(r)-\tfrac{1}{2}(\beta_{0}(r))^{2})dr\\
\text{ \ \ \ \ \ \ \ \ \ \ \ \ \ \ \ \ \ }+%
{\textstyle\int_{t}^{s}}
{\textstyle\int_{\mathbb{R}_{0}}}
(\ln(1+\eta_{0}(r,\zeta))-\eta_{0}(r,\zeta))\nu(d\zeta)dr\\
\text{ \ \ \ \ }+%
{\textstyle\int_{t}^{s}}
{\textstyle\int_{\mathbb{R}_{0}}}
(\ln(1+\eta_{0}(r,\zeta))\widetilde{N}(dr,d\zeta)\}.
\end{array}
\]
and
\[
V=%
{\textstyle\sum_{n=0}^{\infty}}
A^{n}F\,.
\]
\newline We want to maximize the performance functional{\normalsize
\[
J(\pi):=Y(0)=\mathbb{E[}Y(0)\mathbb{]}.
\]
}The corresponding Hamiltonian to this optimization problem $H:[0,T]\times%
\mathbb{R}
^{3}\times L_{\nu}^{2}\times%
\mathbb{R}
^{3}\times\mathcal{U}\times%
\mathbb{R}
^{2}\times L_{\nu}^{2}\times%
\mathbb{R}
\rightarrow%
\mathbb{R}
,$ is defined by%
\[%
\begin{array}
[c]{l}%
H(t,x,y,z,k(\cdot),\overline{y},\overline{z},\overline{k}(\cdot),\pi
,p,q,r(\cdot),\lambda)\\
=\mathbb{(}b_{0}-\pi)xp+\sigma_{0}xq+%
{\textstyle\int_{\mathbb{R}_{0}}}
\gamma_{0}(\zeta)xr(\zeta)\nu(d\zeta)+\lambda\lbrack\alpha_{0}y+\alpha
_{1}\overline{y}\\
+\beta_{0}z+\beta_{1}\overline{z}+%
{\textstyle\int_{\mathbb{R}_{0}}}
\{\eta_{0}(\zeta)k(\zeta)+\eta_{1}(\zeta)\overline{k}(\zeta)\}\nu(d\zeta
)+\ln\pi + \ln x]
\end{array}
\]
where the adjoint processes, for the linear mean field bsde
{\normalsize $(p,q,r)=(p^{\pi},q^{\pi},r^{\pi})$ }and for the linear
differential equation {\normalsize $\lambda=\lambda^{\pi}$ }corresponding to
{\normalsize $\pi,$ }are defined by

{\normalsize
\[
\left\{
\begin{array}
[c]{ll}%
dp(t) & =-[\mathbb{(}b_{0}(t)-\pi(t))p(t)+\sigma_{0}(t)q(t)+%
{\textstyle\int_{\mathbb{R}_{0}}}
\gamma_{0}(t,\zeta)r(t,\zeta)\nu(d\zeta)]dt\\
& +q(t)dB(t)+%
{\textstyle\int_{\mathbb{R}_{0}}}
r(t,\zeta)\tilde{N}(dt,d\zeta),t\in\left[  0,T\right]  ,\\
p(T) & =\theta,
\end{array}
\right.
\]
}and

{\normalsize
\[
\left\{
\begin{array}
[c]{ll}%
d\lambda(t) & =(\alpha_{0}(t)\lambda(t)+\alpha_{1}(t)\mathbb{E[}%
\lambda(t)])dt+(\beta_{0}(t)\lambda(t)+\beta_{1}(t)\mathbb{E[}\lambda
(t)])dB(t)\\
& +%
{\textstyle\int_{\mathbb{R}_{0}}}
(\eta_{0}(t,\zeta)\lambda(t)+\eta_{1}(t,\zeta)\mathbb{E[}\lambda(t)])\tilde
{N}(dt,d\zeta)],t\in\left[  0,T\right]  ,\\
\lambda(0) & =1.
\end{array}
\right.
\]
}Consequently{\normalsize
\begin{equation}%
\begin{array}
[c]{c}%
\lambda(t)=\Upsilon^{-1}(t)[1+%
{\textstyle\int_{0}^{t}}
\{ \Upsilon(r)(\alpha_{1}(r)\mathbb{E[}\lambda(r)]\\
+%
{\textstyle\int_{\mathbb{R}_{0}}}
(\frac{1}{1+\eta_{0}(r,\zeta)}-1)\eta_{1}(r,\zeta)\mathbb{E[}\lambda
(r)]\}\nu(d\zeta)dr\\
+%
{\textstyle\int_{0}^{t}}
\Upsilon(r)\beta_{1}(r)\mathbb{E[}\lambda(r)]dB(r)\\
+%
{\textstyle\int_{0}^{t}}
{\textstyle\int_{\mathbb{R}_{0}}}
\Upsilon(r)\frac{\eta_{1}(r,\zeta)\mathbb{E[}\lambda(r)]}{1+\eta_{0}(r,\zeta
)}\widetilde{N}(dr,d\zeta)\},
\end{array}
\label{Eqlamda}%
\end{equation}
where}%

\[%
\begin{array}
[c]{c}%
\Upsilon(t)=\exp(%
{\textstyle\int_{0}^{t}}
\{-\alpha_{0}(r)+\tfrac{1}{2}\beta_{0}^{2}(r)-%
{\textstyle\int_{\mathbb{R}_{0}}}
\{\ln(1+\eta_{0}(r,\zeta))-\eta_{0}(r,\zeta)\}\nu(d\zeta)dr\\
-%
{\textstyle\int_{0}^{t}}
\beta_{0}(r)dB(r)+%
{\textstyle\int_{0}^{t}}
{\textstyle\int_{\mathbb{R}_{0}}}
\ln(1+\eta_{1}(r,\zeta))\widetilde{N}(dr,d\zeta)\},
\end{array}
\]%
\[
\mathbb{E[}\lambda(r)]=\exp(%
{\textstyle\int_{0}^{t}}
\{\alpha_{0}(r)+\alpha_{1}(r)\}dr),
\]
and{\normalsize
\begin{equation}
p(t)=\mathbb{E}^{\mathcal{F}_{t}}[\theta],t\in\left[  0,T\right]  ,\text{
}\label{Ep}%
\end{equation}
}Now differentiate $H$ with respect to $\pi$, we obtain {\normalsize
\begin{equation}
\tfrac{\partial}{\partial\pi}H(t)=-p(t)+\tfrac{\lambda(t)}{\pi(t)}.\nonumber
\end{equation}
}The first order necessary condition of optimality, yields:\newline

\begin{theorem}
The optimal control $\widehat{{\normalsize \pi}}${\normalsize $(t)$ }is given
by{\normalsize \
\begin{equation}
\widehat{\pi}(t)=\tfrac{\widehat{\lambda}(t)}{\widehat{p}(t)}, \label{maxcond}%
\end{equation}
where }$\widehat{\lambda}(t)$ and $\widehat{p}(t)$ are the solutions of
the equations (\ref{Eqlamda}) and (\ref{Ep}) respectively, corresponding to
the optimal control $\widehat{\pi}(t)$.
\end{theorem}

\bigskip
\section{Appendix}

\subsection{Special case of linear mean-field bsde}

We first define the measure $\mathbb{Q}$ by
\[
d\mathbb{Q}=M(T)d\mathbb{P}\text{ on }\mathcal{F}_{T},
\]
where
\begin{align*}
M(t) &  :=\exp(%
{\textstyle\int_{0}^{t}}
\beta^{1}(s)dB(s)-\tfrac{1}{2}%
{\textstyle\int_{0}^{t}}
(\beta^{1}(s))ds+%
{\textstyle\int_{0}^{t}}
{\textstyle\int_{{\mathbb{R}}_{0}}}
\ln(1+\eta^{1}(s,\zeta))\tilde{N}(ds,d\zeta)\\
&  +%
{\textstyle\int_{0}^{t}}
{\textstyle\int_{{\mathbb{R}}_{0}}}
\{\ln(1+\eta^{1}(s,\zeta))-\eta^{1}(s,\zeta)\}\nu(d\zeta)ds);\quad0\leq t\leq
T.
\end{align*}
Then, under the measure $\mathbb{Q}$ the process
\begin{equation}
B_{\mathbb{Q}}(t):=B(t)-%
{\textstyle\int_{0}^{t}}
\beta^{1}(s)ds\,,\quad0\leq t\leq T\,,\label{e.def_bq}%
\end{equation}
is a Brownian motion, and the random measure
\begin{equation}
\tilde{N}_{\mathbb{Q}}(dt,d\zeta):=\tilde{N}(dt,d\zeta)-\eta^{1}(t,\zeta
)\nu(d\zeta)dt\label{e.def_nq}%
\end{equation}
is the $\mathbb{Q}$-compensated Poisson random measure of $N(\cdot,\cdot)$, in
the sense that the process%

\[
\tilde{N}_{\gamma}(t):=%
{\textstyle\int_{0}^{t}}
{\textstyle\int_{{\mathbb{R}}_{0}}}
\gamma(s,\zeta)\tilde{N}_{\mathbb{Q}}(ds,d\zeta)
\]
is a local $\mathbb{Q}$-martingale, for all predictable processes
$\gamma(t,\zeta)$ such that
\[%
{\textstyle\int_{0}^{T}}
{\textstyle\int_{{\mathbb{R}}_{0}}}
(\gamma(t,\zeta)\eta^{1}(t,\zeta))^{2}\nu(d\zeta)dt<\infty.
\]
Consider the following linear mean field bsde
\begin{equation}
\left\{
\begin{array}
[c]{ll}%
dY(t) & =-[\alpha^{1}(t)Y(t)+\beta^{1}(t)Z(t)+%
{\textstyle\int_{\mathbb{R}_{0}}}
\eta^{1}(t,\zeta)K(t,\zeta)\nu(d\zeta)+\alpha^{2}(t)\mathbb{E}_{\mathbb{Q}%
}\mathbb{[}Y(t)]\\
& +\gamma(t)]dt+Z(t)dB(t)+%
{\textstyle\int_{\mathbb{R}_{0}}}
K(t,\zeta)\tilde{N}(ds,d\zeta),t\in\left[  0,T\right]  ,\\
Y(T) & =\xi,
\end{array}
\right.  \label{p_lbsde}%
\end{equation}
where $\alpha^{1}(t),\alpha^{2}(t),\beta^{1}(t),\eta^{1}(t,\cdot)$ are given
deterministic functions and $\gamma(t)$ is $\mathbb{F}$-adapted and $\xi\in
L^{2}\left(  \Omega,\mathcal{F}_{T}\right)  $. Then, by change of measure, the
linear mean field bsde (\ref{p_lbsde}) is equivallent to
\begin{equation}
\left\{
\begin{array}
[c]{ll}%
dY(t) & =-[\alpha^{1}(t)Y(t)+\alpha^{2}(t)\mathbb{E_{\mathbb{Q}}[}%
Y(t)]+\gamma(s)]dt+Z(t)dB_{\mathbb{Q}}(t)\\
& \text{ \ \ \ \ \ \ \ \ }+%
{\textstyle\int_{\mathbb{R}_{0}}}
K(t,\zeta)\tilde{N}_{\mathbb{Q}}(ds,d\zeta),t\in\left[  0,T\right]  ,\\
Y(T) & =\xi.
\end{array}
\right.  \label{Q_lbsde}%
\end{equation}
Then, $Y(t)$ is given by
\begin{equation}%
\begin{array}
[c]{ll}%
Y(t) & =\mathbb{E}_{\mathbb{Q}}^{\mathcal{F}_{t}}[\xi\Gamma^{\prime}%
(t,T)+\int_{t}^{T}\Gamma^{\prime}(t,s)\{\alpha^{2}(s)\mathbb{E}_{\mathbb{Q}%
}\mathbb{[}Y(s)]+\gamma(s)\}ds],t\in\left[  0,T\right]  ,\text{ a.s.,}%
\end{array}
\label{cf_l_q}%
\end{equation}
for $\Gamma^{\prime}(t,s)=\exp(%
{\textstyle\int_{t}^{s}}
\alpha^{1}(r)dr).$\newline It remains to find $\mathbb{E}_{\mathbb{Q}%
}\mathbb{[}Y(t)]$. Taking the expectation of both sides of (\ref{Q_lbsde}), we
end up with%
\[
\mathbb{E_{\mathbb{Q}}[}Y(t)]=\exp(-%
{\textstyle\int_{0}^{t}}
\{\alpha^{1}(s)+\alpha^{2}(s)\}ds)(Y(0)+%
{\textstyle\int_{0}^{t}}
\tfrac{\mathbb{E}_{\mathbb{Q}}\mathbb{[}\gamma(s)]}{\alpha^{1}(s)}\exp(%
{\textstyle\int_{0}^{s}}
\{\alpha^{1}(r)+\alpha^{2}(r)\}dr)ds),
\]
for some deterministic value $Y(0).$

\end{document}